\documentclass[a4paper,12pt,oneside]{article}

\usepackage{amsmath}
\usepackage{amssymb}
\usepackage{amsthm}

\newtheorem{example}{Example}[section]

\newtheorem{theorem}[example]{Theorem}
\newtheorem{corollary}[example]{Corollary}

\newtheorem{lemma}[example]{Lemma}
\newtheorem{remark}[example]{Remark}

\title{\textbf{On Intersection Graph of Dihedral Group}} 
\author{Sanhan M.S. ~Khasraw\\
Department of Mathematics, College of Basic Education,\\ Salahaddin University-Erbil, Erbil, Kurdistan Region, Iraq\\
sanhan.khasraw@su.edu.krd}
\date{}

\begin{document}
\maketitle

\begin{center}
{\large{\textbf{Abstract}}}\\
\end{center}

Let $G$ be a finite group. The intersection graph of $G$ is a graph whose vertex set is the set of all proper non-trivial subgroups of $G$ and two distinct vertices $H$ and $K$ are adjacent if and only if $H\cap K \neq \{e\}$, where $e$ is the identity of the group $G$. In this paper, we investigate some properties and exploring some topological indices such as Wiener, Hyper-Wiener, first and second Zagreb, Schultz, Gutman and eccentric connectivity indices of the intersection graph of $D_{2n}$ for $n=p^2$, $p$ is prime. We also find the metric dimension and the resolving polynomial of the intersection graph of $D_{2p^2}$.

\hspace{1cm}\\
	\text{\bf Keywords}: Intersection graph of subgroups, Wiener index, Zagreb indices, Schultz index, resolving polynomial of a graph.	

\section{Introduction}

The notion of intersection graph of a finite group has been introduced by Cs{\'a}k{\'a}ny and Poll{\'a}k in 1969 \cite{intersection}. For a finite group $G$, associate a graph $\Gamma(G)$ with it in such away that the set of vertices of $\Gamma(G)$ is the set of all proper non-trivial subgroups of $G$ and join two vertices if their intersection is non-trivial. For more studies about intersection graphs of subgroups and related topics, we refer the reader to see \cite{zelinka1975intersection, chakrabarty2009intersection, shen2010intersection, rajkumara4intersection, RAJKUMAR201615, sehgal1, maan, sehgal2}.

Suppose that $\Gamma$ is a simple graph, which is undirected and contains no multiple edges or loops. We denote the set of vertices of $\Gamma$ by $V(\Gamma)$ and the set of edges of $\Gamma$ by $E(\Gamma)$. We write $uv \in E(\Gamma)$ if $u$ and $v$ form an edge in $\Gamma$. The size of the vertex-set of $\Gamma$ is denoted by $|V(\Gamma)|$ and the number of edges of $\Gamma$ is denoted by $|E(\Gamma)|$. The \textit{degree} of a vertex $v$ in $\Gamma$, denoted by $deg(v)$, is defined as the number of edges incident to $v$. The \textit{distance} between any pair of vertices $u$ and $v$ in $\Gamma$, denoted by $d(u, v)$, is the shortest $u-v$ path in $\Gamma$. For a vertex $v$ in $\Gamma$, the \textit{eccentricity} of $v$, denoted by $ecc(v)$, is the largest distance between $v$ and any other vertex in $\Gamma$. The \textit{diameter} of $\Gamma$, denoted as $diam(\Gamma)$, is defined by $diam(\Gamma)=max\{ecc(v) : v \in V(\Gamma)\}$. A graph $\Gamma$ is called \textit{complete} if every pair of vertices in $\Gamma$ are adjacent. If $S \subseteq V(\Gamma)$ and no two elements of $S$ are adjacent, then $S$ is called an \textit{independent set}. The cardinality of the largest independent set is called an \textit{independent number} of the graph $\Gamma$. A graph $\Gamma$ is called \textit{bipartite} if the set $V(\Gamma)$ can be partitioned into two disjoint independent sets such that each edge in $\Gamma$ has its ends in different independent sets. A graph $\Gamma$ is called \textit{split} if $V(\Gamma)$ can be partitioned into two different sets $U$ and $K$ such that $U$ is an independent set and the subgraph induced by $K$ is a complete graph.

Let $W=\{v_1, v_2, \cdots, v_k\} \subseteq V(\Gamma)$ and let $v$ be any vertex of $\Gamma$. The \textit{representation} of $v$ with respect to $W$ is the k-vector $r(v|W) = (d(v, v_1), d(v, v_2),$ $ \cdots, d(v, v_k))$. If distinct vertices have distinct representations with respect to $W$, then $W$ is called a \textit{resolving set} for $\Gamma$. A \textit{basis} of $\Gamma$ is a minimum resolving set for $\Gamma$ and the cardinality of a basis of $\Gamma$ is called the \textit{metric dimension} of $\Gamma$ and denoted by $\beta(\Gamma)$ \cite{metdim}. Suppose $r_i$ is the number of resolving sets for $\Gamma$ of cardinality $i$. Then the \textit{resolving polynomial} of a graph $\Gamma$ of order $n$, denoted by $\beta(\Gamma, x)$, is defined as $\beta(\Gamma, x)=\sum_{i=\beta(\Gamma)}^{n}r_ix^i$. The sequence $(r_{\beta(\Gamma)}, r_{\beta(\Gamma)+1}, \cdots, r_n)$ formed from the coefficients of $\beta(\Gamma, x)$ is called the \textit{resolving sequence}.

For a graph $\Gamma$, the \textit{Wiener index} is defined by $W(\Gamma) = \sum_{\{u, v\} \subseteq V(\Gamma)}d(u, v)$ \cite{wiener}. The \textit{hyper-Wiener index} of $\Gamma$ is defined by\\ $ WW(\Gamma) = \frac{1}{2}W(\Gamma) +  \frac{1}{2}\sum_{\{u, v\} \subseteq V(\Gamma)}(d(u, v))^2 $ \cite{hyperwiener}. The \textit{Zagreb indices} are defined by $M_1(\Gamma)=\sum_{v \in V(\Gamma)}(deg(v))^2$ and $M_2(\Gamma)=\sum_{uv \in E(\Gamma)}deg(u)deg(v)$ \cite{zagreb}. The \textit{Schultz index} of $\Gamma$, denoted by $MTI(\Gamma)$ is defined in \cite{schultz} by $MTI(\Gamma)=\sum_{\{u, v\} \subseteq V(\Gamma)} d(u,v) [deg(u) + deg(v)]$. In \cite{gutman1, gutman2} the \textit{Gutman index} has been defined by $Gut(\Gamma)=\sum_{\{u, v\} \subseteq V(\Gamma)}d(u,v)[deg(u)\times deg(v)]$. Sharma, Goswami and
Madan defined the \textit{eccentric connectivity index} of $\Gamma$, denoted by $\xi^c(\Gamma)$, in \cite{sharma} by $\xi^c(\Gamma)=\sum_{v \in V(\Gamma)}deg(v)ecc(v)$. 

For an integer $n \geq 3$, the dihedral group $D_{2n}$ of order $2n$ is defined by  
$$D_{2n}=\langle r, s : r^n = s^2 =1, srs = r^{-1} \rangle.$$

In \cite{rajkumara4intersection}, Rajkumar and Devi studied the intersection graph of subgroups of some non-abelian groups, especially the dihedral group $D_{2n}$, quaternion group $Q_n$ and quasi-dihedral group $QD_{2^\alpha}$. They were only able to obtain the clique number and degree of vertices. It seems difficult to study most properties of the intersection graph of subgroups of these groups. In this paper, the focus will be on the intersection graph of subgroups of the dihedral group $D_{2n}$ for the case when $n=p^2$, $p$ is prime. It is clear that when $n=p$, then the resulting intersection graph of subgroups is a null graph, which is not of our interest. For $n=p^2$, the intersection graph $\Gamma({D_{2p^2}})$ of the group $D_{2p^2}$ has $p^2+p+2$ vertices. We leave the other possibilities for $n$ open and we might be able to work on them in the future. So, all throughout this paper, the considered dihedral group is of order $2p^2$, and by intersection graph we mean intersection graph of subgroups.
 
This paper is organized as follows. In Section 2, some basic properties of the intersection graph of $D_{2p^2}$ are presented. We see that the intersection graph $\Gamma({D_{2p^2}})$ is split. In Section 3, we find some topological indices of the intersection graph $\Gamma({D_{2p^2}})$ of $D_{2p^2}$ such as the Wiener, hyper-Wiener and Zagreb indices. In Section 4, we find the metric dimension and the resolving polynomial of the intersection graph $\Gamma({D_{2p^2}})$.

\section{Some properties of the intersection graph of $D_{2n}$}

In \cite{rajkumara4intersection}, all proper non-trivial subgroups of the group $D_{2n}$ has been classified as shown in the following lemma.
 
\begin{lemma}\label{lem21}
The proper non-trivial subgroups of $D_{2n}$ are:
\begin{enumerate}
\item cyclic groups $H^k=\langle r^{\frac{n}{k}}\rangle$ of order $k$, where $k$ is a divisor of $n$ and $k\neq 1$,
\item cyclic groups $H_i=\langle sr^i \rangle$ of order 2, where $i=1, 2, \cdots, n$, and 
\item dihedral groups $H_k^i=\langle r^{\frac{n}{k}}, sr^i \rangle$ of order $2k$, where $k$ is a divisor of $n$, $k \neq 1, n$ and $i=1, 2, \cdots, \frac{n}{k}$.
\end{enumerate}
\end{lemma}
 
The total number of these proper subgroups is $\tau(n)+\sigma(n)-2$, where $\tau(n)$ is the number of positive divisors of $n$ and $\sigma(n)$ is the sum of positive divisors of $n$. We mentioned that we only focus on the case when $n=p^2$, $p$ is prime. Recall that, for $n=p^2$, the intersection graph $\Gamma({D_{2p^2}})$ of the group $D_{2p^2}$ has $p^2+p+2$ vertices. The vertex set of $\Gamma({D_{2p^2}})$ is $V(\Gamma({D_{2p^2}}))=(\cup_{i=1}^{p^2} \{H_i\})\cup (\cup_{i=1}^p \{H_{p}^{i}\}) \cup \{H^p\}\cup \{H^{p^2}\}$, where  
\begin{enumerate}
\item $H_i=\langle sr^{i} \rangle$, where $i=1, 2, \cdots, p^2$, 
\item $H_{p}^{i}=\langle r^p, sr^i\rangle$,  where $i=1, 2, \cdots, p$, 
\item $H^p=\langle r^p \rangle$ and $H^{p^2}=\langle r \rangle$.
\end{enumerate}    

 \medskip

The following theorem is given in \cite{rajkumara4intersection} to compute the degree of any vertex in $\Gamma({D_{2n}})$. Since we only consider the case $n=p^2$, we restate it as follows: 

 \begin{theorem}\label{thm22} 
 In the graph $\Gamma({D_{2p^2}})$, $$ deg(v)=\left\{
 \begin{tabular}{ll}
 $1$, & \mbox{ if }$v = H_i$ \mbox{ for $i=1, 2, \cdots, p^2$ },\\
 $2p+1$ & \mbox{ if }$v = H_{p}^{i}$ \mbox{ for $i=1, 2, \cdots, p$ },\\
 $p+1$, & \mbox{ if }$v = H^p \;\; or\;\; H^{p^2}$. \\
 \end{tabular}
 \right.
 $$
 \end{theorem}

\medskip
 
 The following theorem gives the exact number of edges in $\Gamma({D_{2p^2}})$ which can be in the Section 3 to compute the second Zagreb index.
 
 \begin{theorem}\label{thm211}
 In the graph $\Gamma({D_{2p^2}})$, $|E(\Gamma(D_{2p^2}))|=\frac{1}{2}(3p^2+3p+2)$.
 \end{theorem}
 \begin{proof}
 It follows from Theorem \ref{thm22} that there are $p^2$ vertices of degree 1, $p$ vertices of degree $2p+1$ and 2 vertices of degree $p+1$. Thus, $|E(\Gamma(D_{2p^2}))| =\frac{1}{2}\sum_{v\in V(\Gamma({D_{2p^2}}))}deg(v) = \frac{1}{2}(p^2 \cdot 1 + p \cdot(2p+1) + 2 \cdot (p+1)) = \frac{1}{2}(3p^2+3p+2)$. 
 \end{proof}
 
\medskip

\begin{theorem}\label{thm23}
Let $\Gamma=\Gamma({D_{2p^2}})$ be an intersection graph on $D_{2p^2}$. Then $diam(\Gamma)=3$. In particular, $\Gamma$ is connected.  
\end{theorem}
\begin{proof}
Suppose $u$ and $v$ are two distinct vertices of $\Gamma({D_{2p^2}})$. If $u$ and $v$ are adjacent, then $d(u, v)=1$. Otherwise, let $u \cap v=\{e\}$. Then there are three possibilities: 
$u=H_i$ and $v=H_j$ for $i\neq j$, $u=H_i$ and $v=H^j$ for $j=p$ or $p^2$, and $u=H_i$ and $v=H_p^j$ for $i \not\equiv j (mod p)$. For the first case, if $i \equiv j (mod p)$, then there exists $w=H_p^k$ where $k\equiv i (mod p)$ such that $uw, wv \in E(\Gamma)$ and so $d(u, v)=2$. But if $i \not\equiv j (mod p)$, then no such $w$ exist such that $uw, vw \in E(\Gamma)$. Then take $w=H_p^{k_1}$ where $k_1\equiv i (mod p)$ and $w'=H_p^{k_2}$ where $k_2\equiv j (mod p)$, and so $uw, ww', w'v \in E(\Gamma)$. Hence $d(u, v)=3$. For the second case, there exists $w=H_p^k$, where $k\equiv i (mod p)$, such that $uw, wv \in E(\Gamma)$. Hence $d(u, v)=2$. For the last case, there exists $w=H_p^k$, where $k\equiv i (mod p)$, such that $uw, wv \in E(\Gamma)$ and then $d(u, v)=2$.  
\end{proof}

\medskip

From Theorem \ref{thm23}, one can see that the maximum distance between any pair of vertices in $\Gamma(D_{2p^2})$ is 3. In order to explore the exact distance between any pair of vertices in $\Gamma(D_{2p^2})$ , we state the following corollary which can be used in the next section to find some topological indices of $\Gamma(D_{2p^2})$. 

\begin{corollary}\label{cor24} 
In the graph $\Gamma({D_{2p^2}})$, $$ d(u, v)=\left\{
\begin{tabular}{ll}

$1$ & \mbox{ if }$u= H_i, v=H_{p}^{j}$ \mbox{ where} $i\equiv j (mod p)$ \mbox{ for $i=1, 2, \cdots, p^2$}\\
    & \mbox{and $j=1, 2, \cdots, p$},\\

$1$ & \mbox{ if }$u=H^p$ \mbox{or} $H^{p^2}, v=H_p^j$ \mbox{ where $j=1, 2, \cdots, p$}, \\

$1$ & \mbox{ if }$u=H^p, v=H^{p^2}$, \\

$1$ & \mbox{ if }$u= H_p^j, v= H_p^l $ \mbox{ where} $j\neq l$ \mbox{ and $j, l=1, 2, \cdots, p$}, \\

$2$ & \mbox{ if }$u= H_i, v= H^p$ \mbox{or} $H^{p^2}$ \mbox{ for $i=1, 2, \cdots, p^2$},\\

$2$ & \mbox{ if }$u= H_i, v=H_{p}^{j}$ \mbox{ where} $i\not\equiv j (mod p)$ \mbox{ for $i=1, 2, \cdots, p^2$ }\\
    & \mbox{and $j=1, 2, \cdots, p$},\\

$2$ & \mbox{ if }$u= H_i, v= H_j $ \mbox{where} $i\neq j$ \mbox{and} $i\equiv j (mod p)$ \\

    & \mbox{ for $i, j=1, 2, \cdots, p^2$}, and\\

$3$ & \mbox{ if }$u= H_i, v= H_j $ \mbox{where} $i\not\equiv j (mod p)$ \mbox{ for $i, j=1, 2, \cdots, p^2$}\\
\end{tabular}
\right.
$$

\end{corollary}

\medskip
  
\begin{theorem}\label{thm25}
Let $\Gamma=\Gamma({D_{2p^2}})$ be an intersection graph on $D_{2p^2}$. Then $\cup_{i=1}^{p^2}\{H_i\}$ is an independent set.
\end{theorem}
\begin{proof}
From Corollary \ref{cor24}, $d(u, v)\neq 1$ for every distinct pairs of vertices $u, v \in \cup_{i=1}^{p^2}\{H_i\}$ and so $uv \notin E(\Gamma)$. Therefore, $\cup_{i=1}^{p^2}\{H_i\}$ is an independent set for each $i$.   
\end{proof}

\medskip

\begin{corollary}\label{cor26}
The independent number of the graph $\Gamma({D_{2p^2}})$ is $p^2+1$.
\end{corollary}
\begin{proof}
From Theorem \ref{thm25}, the independent set $\cup_{i=1}^{p^2}\{H_i\}$ is of size $p^2$. Also, from Corollary \ref{cor24}, one can see that none of the vertices of $H^p$ or $H^{p^2}$ is adjacent to vertices in $\cup_{i=1}^{p^2}\{H_i\}$. So, in total the size of the largest independent set is $p^2+1$. 
\end{proof}

\medskip

\begin{theorem}\label{thm27}
Let $H\subseteq V(\Gamma({D_{2p^2}}))$. Then the intersection graph $\Gamma(H)$ is complete if and only if $H=\cup_{i=1}^p \{H_{p}^{i}\} \cup \{H^p\} \cup \{H^{p^2}\}$.
\end{theorem}
\begin{proof}
 Suppose $H=\cup_{i=1}^p \{H_{p}^{i}\} \cup \{H^p\} \cup \{H^{p^2}\}$. By Corollary \ref{cor24}, $d(u, v)=1$ for every distinct pairs of vertices $u, v \in H$. Then the graph $\Gamma(H)$ is complete. The converse follows directly from Corollary \ref{cor24}.  
\end{proof}

\medskip

The complete graph in the previous theorem is the largest complete subgraph of $\Gamma({D_{2n}})$. As a consequence, the clique number of $\Gamma({D_{2n}})$ is $p+2$ which coincides with Theorem 2.3 in \cite{rajkumara4intersection}.

\medskip

\begin{theorem}\label{thm28}
Let $H\subseteq V(\Gamma({D_{2p^2}}))$. Then $\Gamma(H)=K_{1, p}$ if and only if $H=\cup_{i=1}^p \{H_i\} \cup \{H_{p}^{j}\}$ where $i \equiv j (mod p)$.
\end{theorem}
\begin{proof}
 The proof follows from Theorems \ref{thm25} and \ref{thm27}.
\end{proof}

\medskip

As a consequence of the above theorem, we have the following corollary.

\begin{corollary}\label{cor29}
The graph $\Gamma({D_{2p^2}})$ is split.
\end{corollary}

\medskip

\begin{theorem}\label{thm210}
 In the graph $\Gamma({D_{2p^2}})$, $$ ecc(v)=\left\{
 \begin{tabular}{ll}
 $3$ & \mbox{ if }$v= H_i$ \mbox{ for $i=1, 2, \cdots, p^2$} \\
 $2$ & \mbox{ otherwise}.\\
 \end{tabular}
 \right.
 $$
 \end{theorem}
 \begin{proof}
 
Let $v=H_i$ for some $i$. By Corollary \ref{cor24}, $d(u, v)=3$ if $u=H_j$ where $i\not\equiv j (mod p)$, otherwise $d(u, v)< 3$. Thus, $ecc(v)=3$ for every $v\in \cup_{i=1}^{p^2} \{H_i\}$. If $v \neq H_i$ for any $i$, then again from Corollary \ref{cor24}, the maximum distance between $v$ and any other vertex is 2, and so $ecc(v)=2$ for each $v\notin \cup_{i=1}^{p^2} \{H_i\}$.   
 \end{proof}
 
\section{Some Topological Indices of intersection graph on $D_{2p^2}$}

In this section, some topological indices, such as the Wiener index, Hyper-Wiener index, Zagreb indices, the Schultz index, the Gutman index and the eccentric connectivity index, of the intersection graph for the dihedral group $D_{2n}$, where $n=p^2$, are computed.  

\begin{theorem}\label{thm31}
Let $\Gamma=\Gamma({D_{2n}})$ be an intersection graph on $D_{2n}$. Then
$$W(\Gamma)=\frac{1}{2}(3p^4+3p^3+5p^2+3p+2).$$
\end{theorem}
\begin{proof}
Let $u, v \in V(\Gamma)$. It follows from Corollary \ref{cor24} that the number of possibilities of $d(u, v)=1$ is $p^2 + {{p+2}\choose{2}}$, the number of possibilities of $d(u, v)=2$ is $p\cdot{{p}\choose{2}} + p\cdot p\cdot (p+1)$ and the number of possibilities of $d(u, v)=3$ is ${{p}\choose{2}}{{p}\choose{1}}{{p}\choose{1}}$. Thus, 
$W(\Gamma({D_{2n}}))=(p^2+\frac{1}{2}(p+1)(p+2))\cdot 1 + (\frac{1}{2}(3p^3+p^2)) \cdot 2 + (\frac{1}{2}(p^4-p^3)) \cdot 3=\frac{1}{2}(3p^4+3p^3+5p^2+3p+2)$.
\end{proof}

\medskip

\begin{theorem}\label{thm33}
Let $\Gamma({D_{2n}})$ be an intersection graph on $D_{2n}$. Then
$$WW(\Gamma({D_{2n}})) = \frac{1}{2}(6p^4+3p^3+6p^2+3p+2).$$
\end{theorem}
\begin{proof}
From Theorem \ref{thm31} and Corollary \ref{cor24}, we can see that $WW(\Gamma({D_{2n}})) = \frac{1}{2}\bigg(\frac{1}{2}(3p^4+3p^3+5p^2+3p+2)\bigg) + \frac{1}{2}\bigg(\bigg(p^2+\frac{1}{2}(p+1)(p+2)\bigg)\cdot 1^2 + \bigg(\frac{1}{2}(3p^3+p^2)\bigg) \cdot 2^2 + \bigg(\frac{1}{2}(p^4-p^3)\bigg) \cdot 3^2\bigg) = \frac{1}{2}(6p^4+3p^3+6p^2+3p+2)$. 
\end{proof}

\medskip

In the next two theorems, the first and second Zagreb indices for the intersection graph $\Gamma({D_{2n}})$ are presented.
 
\begin{theorem}\label{thm34}
Let $\Gamma({D_{2n}})$ be an intersection graph on $D_{2n}$. Then
$$M_1(\Gamma({D_{2n}}))= 4p^3+7p^2+5p+2.$$
\end{theorem}
\begin{proof}
The proof is similar to the proof of Theorem \ref{thm211}. It follows from Theorem \ref{thm22} that $M_1(\Gamma({D_{2n}}))= p^2 \cdot 1^2 + p \cdot (2p+1)^2 + 2 \cdot (p+1)^2=4p^3+7p^2+5p+2$.   
\end{proof}
 
\medskip

\begin{theorem}
Let $\Gamma({D_{2n}})$ be an intersection graph on $D_{2n}$. Then
$$M_2(\Gamma({D_{2n}})) = 2p^4+6p^3 + \frac{13}{2}p^2 + \frac{7}{2}p+1.$$
\end{theorem}
\begin{proof}
By Theorem \ref{thm211}, $\Gamma$ has $\frac{1}{2}(3p^2+3p+2)$ edges in which $p^2$ edges with one end-vertex of degree 1 and the other end-vertex of degree $2p+1$, $\frac{p(p-1)}{2}$ edges where end-vertices have degree $2p+1$, $2p$ edges with one end-vertex of degree $2p+1$ and the other end-vertex of degree $p+1$ and one edge where end-vertices have degree $p+1$. Thus, $M_2(\Gamma({D_{2n}}))= p^2 \cdot (1)(2p+1) + \frac{p(p-1)}{2} \cdot (2p+1)^2 + 2p \cdot  (2p+1)(p+1) + 1 \cdot (p+1)^2 = 2p^4+6p^3 + \frac{13}{2}p^2 + \frac{7}{2}p+1$.  
\end{proof}

\medskip

\begin{theorem}
Let $\Gamma({D_{2n}})$ be an intersection graph on $D_{2n}$. Then
$$MTI(\Gamma({D_{2n}}))= 7p^4+6p^3+5p^2+5p+2.$$
\end{theorem}
\begin{proof}
By Theorem \ref{thm22} and Corollary \ref{cor24}, 
\begin{align*}
MTI(\Gamma({D_{2n}}))
&=\bigg(\sum_{\substack{u = H_i, v = H_{p}^{j}, i\equiv j (mod p) \\ i=1, 2, \cdots, p^2; j=1, 2, \cdots, p}} d(u,v) [deg(u) + deg(v)] &&\\
&+\sum_{u, v\in \cup_{j=1}^p\{H_{p}^{j}\}} d(u,v) [deg(u) + deg(v)] &&\\
&+\sum_{u, v\in \{H^p, H^{p^2}\}} d(u,v) [deg(u) + deg(v)] &&\\
&+\sum_{u\in \{H^p, H^{p^2}\}, v\in \cup_{j=1}^p\{H_{p}^{j}\}} d(u,v) [deg(u) + deg(v)]\bigg) &&\\
&+\bigg(\sum_{\substack{u=H_i, v\in H_j, i\equiv j (mod p)\\ i, j=1, 2, \cdots, p^2}} d(u,v) [deg(u) + deg(v)]  &&\\
&+\sum_{\substack{u= H_i, v = H_{p}^{j}, i\not\equiv j (mod p)\\ i=1, 2, \cdots, p^2; j=1, 2, \cdots, p}} d(u,v) [deg(u) + deg(v)] &&\\
&+\sum_{\substack{u= H_i, v\in \{H^p, H^{p^2}\}\\i=1, 2, \cdots, p^2}} d(u,v) [deg(u) + deg(v)]\bigg) &&\\
&+\bigg(\sum_{\substack{u= H_i, v= H_j, i\not\equiv j (mod p)\\ i, j=1, 2, \cdots, p^2}} d(u,v) [deg(u) + deg(v)]\bigg)&&\\
&= \bigg(p^2 \cdot 1 \cdot [1 + (p + 1)]+{{p}\choose{2}} \cdot 1 \cdot [(2p + 1) + (2p + 1)]&&\\
&+ 1 \cdot 1 \cdot [(p + 1) + (p + 1)] &&\\
&+{{2}\choose{1}} \cdot {{p}\choose{1}} \cdot 1 \cdot [(p + 1) + (2p + 1)]\bigg)+\bigg(p \cdot {{p}\choose{2}} \cdot 2 \cdot [1 + 1] &&\\
&+ p \cdot p \cdot (p-1) \cdot 2 \cdot [1 + (2p + 1)] &&\\
&+ p \cdot p \cdot 2 \cdot 2 \cdot [1 + (p + 1)]\bigg)+\bigg({{p}\choose{1}} \cdot {{p}\choose{1}} \cdot {{p}\choose{2}} \cdot 3 \cdot [1 + 1]\bigg)&&\\
& =7p^4 + 6p^3 + 5p^2 + 5p + 2.&&
\end{align*}
\end{proof}

\medskip

\begin{theorem}
Let $\Gamma({D_{2n}})$ be an intersection graph on $D_{2n}$. Then
$$Gut(\Gamma({D_{2n}}))= \frac{1}{2}(15p^4 + 13p^3 + 15p^2 + 7p + 2).$$
\end{theorem}
\begin{proof} Again by Theorem \ref{thm22} and Corollary \ref{cor24}, 
\begin{align*}
Gut(\Gamma({D_{2n}}))
&=\bigg(\sum_{\substack{u = H_i, v = H_{p}^{j}, i\equiv j (mod p) \\ i=1, 2, \cdots, p^2; j=1, 2, \cdots, p}} d(u,v) [deg(u) \times deg(v)] &&\\
&+\sum_{u, v\in \cup_{j=1}^p\{H_{p}^{j}\}} d(u,v) [deg(u) \times deg(v)] &&\\
&+\sum_{u, v\in \{H^p, H^{p^2}\}} d(u,v) [deg(u) \times deg(v)] &&\\
&+\sum_{u\in \{H^p, H^{p^2}\}, v\in \cup_{j=1}^p\{H_{p}^{j}\}} d(u,v) [deg(u) \times deg(v)]\bigg) &&\\
&+\bigg(\sum_{\substack{u=H_i, v\in H_j, i\equiv j (mod p)\\ i, j=1, 2, \cdots, p^2}} d(u,v) [deg(u) \times deg(v)]  &&\\
&+\sum_{\substack{u= H_i, v = H_{p}^{j}, i\not\equiv j (mod p)\\ i=1, 2, \cdots, p^2; j=1, 2, \cdots, p}} d(u,v) [deg(u) \times deg(v)] &&\\
&+\sum_{\substack{u= H_i, v\in \{H^p, H^{p^2}\}\\i=1, 2, \cdots, p^2}} d(u,v) [deg(u) \times deg(v)]\bigg) &&\\
&+\bigg(\sum_{\substack{u= H_i, v= H_j, i\not\equiv j (mod p)\\ i, j=1, 2, \cdots, p^2}} d(u,v) [deg(u) \times deg(v)]\bigg)&&\\
&= \bigg(p^2 \cdot 1 \cdot [1 \times (p + 1)]+{{p}\choose{2}} \cdot 1 \cdot [(2p + 1) \times (2p + 1)] &&\\
&+ 1 \cdot 1 \cdot [(p + 1) \times (p + 1)] &&\\
&+{{2}\choose{1}} \cdot {{p}\choose{1}} \cdot 1 \cdot [(p + 1) \times (2p + 1)]\bigg) &&\\
&+\bigg(p \cdot {{p}\choose{2}} \cdot 2 \cdot [1 \times 1]+ p \cdot p \cdot (p-1) \cdot 2 \cdot [1 \times (2p + 1)] &&\\
&+ p \cdot p \cdot 2 \cdot 2 \cdot [1 \times (p + 1)]\bigg)+\bigg({{p}\choose{1}} \cdot {{p}\choose{1}} \cdot {{p}\choose{2}} \cdot 3 \cdot [1 \times 1]\bigg)&&\\
& =\frac{1}{2}(15p^4 + 13p^3 + 15p^2 + 7p + 2).&&
\end{align*}
\end{proof}

\medskip

\begin{theorem}
Let $\Gamma({D_{2n}})$ be an intersection graph on $D_{2n}$. Then
$$\xi^c(\Gamma({D_{2n}}))= 7p^2 + 6p + 4.$$
\end{theorem}
\begin{proof} By Theorems \ref{thm22} and \ref{thm210}, we see that
\begin{align*}
&\xi^c(\Gamma({D_{2n}}))\\
& =\sum_{v \in \cup_{i=1}^{p^2} \{H_i\}}deg(v)ecc(v) + \sum_{v \in \cup_{j=1}^p \{H_p^j\}}deg(v)ecc(v) + \sum_{v\in \{H^p, H^{p^2}\}}deg(v)ecc(v) &&\\
& =\sum_{i=1}^{p^2} 1 \times 3 + \sum_{j=1}^{p}(2p+1)\times 2 + \sum_{k=1}^{2}(p+1)\times 2 &&\\
& =p^2 \times 1 \times 3 + p \times (2p+1)\times 2 + 2 \times (p+1)\times 2 &&\\
& =7p^2 + 6p + 4. &&
\end{align*}
\end{proof}

\section{Metric dimension and resolving polynomial of intersection graph on $D_{2p^2}$}

For a vertex $u$ of a graph $\Gamma$, the set $N(u)=\{v\in V(\Gamma) : uv\in E(\Gamma)\}$ is called the \textit{open neighborhood} of $u$ and the set $N[u]=N(u)\cup\{u\}$ is called the \textit{closed neighborhood} of $u$. If $u$ and $v$ are two distinct vertices of $\Gamma$, then $u$ and $v$ are said to be \textit{adjacent twins} if $N[u]=N[v]$ and \textit{non-adjacent twins} if $N(u)=N(v)$. Two distinct vertices are called \textit{twins} if they are adjacent or non-adjacent twins. A subset $U \subseteq V(\Gamma)$ is called a \textit{twin-set} in $\Gamma$ if every pair of distinct vertices in $U$ are twins.  

\medskip

\begin{lemma}\label{lem41}
Let $\Gamma$ be a connected graph of order $n$ and $U \subseteq V(\Gamma)$ be a twin set in $\Gamma$ with $|U|=m$. Then every resolving set for $\Gamma$ contains at least $m-1$ vertices of $U$.
\end{lemma}

\medskip

\begin{corollary}\textsl{\cite{hernando2010extremal}}\label{cor42}
Let $\Gamma$ be a connected graph, $U$ resolves $\Gamma$ and $u$ and $v$ are twins. Then $u\in U$ or $v\in U$. In addition, if $u\in U$ and $v \notin U$, then $(U\setminus \{u\})\cup\{v\}$ also resolves $\Gamma$.
\end{corollary}

\medskip

\begin{theorem}\label{thm42}
Let $\Gamma(D_{2p^2})$ be an intersection graph on $D_{2p^2}$. Then $$\beta(\Gamma(D_{2p^2}))=p^2-p+1.$$
\end{theorem}
\begin{proof}
Let $W=((\cup_{i=1}^{p^2} \{H_i\})\cup\{H^p\}) - S$, where $S=\{H_1, H_2, \cdots, H_p\}$ with the property that $H_i$ and $H_j$ are in $S$ if and only if $i\not\equiv j (mod p)$. One can see that $W$ is a resolving set for $\Gamma(D_{2p^2})$ of cardinality $p(p-1)+1$. Then $\beta(\Gamma({D_{2p^2}}))\le p^2-p+1$. On the other hand, $\cup_{i=1}^{p^2} \{H_i\}$ is the union of $p$ twin sets each of cardinality $p$ such that $H_i$ and $H_j$ belong to the same set if and only if $i\equiv j (mod p)$. Also, $\{H^p, H^{p^2}\}$ is a twin set of cardinality 2. Then by Lemma \ref{lem41}, we see that $\beta(\Gamma({D_{2p^2}}))\geq p(p-1)+1$.  
\end{proof}

\medskip

The following is a useful property for finding a resolving polynomial of a graph of order $n$.

\medskip

\begin{lemma}\label{lem43}
If $\Gamma$ is a connected graph of order $n$, then $r_n=1$ and $r_{n-1}=n$.
\end{lemma}

\medskip

\begin{theorem}\label{thm45}
Let $\Gamma=\Gamma(D_{2p^2})$ be an intersection graph on $D_{2p^2}$. Then
\\$\beta(\Gamma, x) = x^{p^2-p+1}\bigg({{2}\choose{1}}{{p}\choose{p-1}}^p + \sum_{q=1}^{p}r_{p^2-p+1+q}x^q + \sum_{k=p+1}^{2p-1}
r_{p^2-p+1+k}x^k+(p^2+p+1)x^{2p} + x^{2p+1}\bigg),$
\\where
\\ $r_{p^2-p+1+q} = {{p}\choose{i}} {{p}\choose{p-1}}^{p-i} {{2}\choose{1}} {{p}\choose{j}} + {{p}\choose{i-1}} {{p}\choose{p-1}}^{p-(i-1)} {{2}\choose{2}} {{p}\choose{j}} + {{p}\choose{i}} {{p}\choose{p-1}}^{p-i} {{2}\choose{2}} {{p}\choose{j-1}}$; $q=i+j$,\\

$r_{p^2-p+1+k} = {{p}\choose{k_1}} {{p}\choose{p-1}}^{p-k_1} {{2}\choose{1}} {{p}\choose{k_2}} + {{p}\choose{k_2}} {{p}\choose{p-1}}^{p-k_2} {{2}\choose{1}} {{p}\choose{k_1}} +\\  {{p}\choose{k_1-1}} {{p}\choose{p-1}}^{p-(k_1-1)} {{2}\choose{2}} {{p}\choose{k_2}} + {{p}\choose{k_1}} {{p}\choose{p-1}}^{p-k_1} {{2}\choose{2}} {{p}\choose{k_2-1}} + {{p}\choose{k_2-1}} {{p}\choose{p-1}}^{p-(k_2-1)} {{2}\choose{2}} {{p}\choose{k_1}}\\ + {{p}\choose{k_2}} {{p}\choose{p-1}}^{p-k_2} {{2}\choose{2}} {{p}\choose{k_1-1}}$,\\

$k_1+k_2=k,k_1 \neq k_2, k_1-1 \neq k_2, k_1 \neq k_2-1$ and $1 \le k_j \le p$ for $j=1, 2$.
\end{theorem}
\begin{proof}
By Theorem \ref{thm42}, $\beta(\Gamma)=p^2-p+1$. It is required to find the resolving sequence $(r_{\beta(\Gamma)}, r_{\beta(\Gamma)+1}, \cdots, r_{\beta(\Gamma)+2p+1})$ of length $2p+2$.\\
To find $r_{\beta(\Gamma)}$. For the reason that $\cup_{i=1}^{p^2}\{H_i\}$ is the union of $p$ twin sets and $\{H^p, H^{p^2}\}$ is also a twin set, then by Corollary \ref{cor42} and the principal of multiplication, we see that there are 

$$\underbrace{{{p}\choose{p-1}}{{p}\choose{p-1}}\cdots {{p}\choose{p-1}}}_{p-times}{{2}\choose{1}}=2p^p$$ possibilities of resolving sets of cardinality $\beta(\Gamma)$, that is, $r_{\beta(\Gamma)}=2p^p$.\\
For $1 \leq l \leq 2p-1$, we aim to find $r_{\beta(\Gamma)+l}$.\\ 
First, we try to find $r_{\beta(\Gamma)+q}$, where $1 \leq q \leq p$. Suppose $u_1, u_2, \cdots, u_q$ be $q$ distinct vertices of $\Gamma$ that do not belong to any resolving set of cardinality $\beta(\Gamma)+q-1$. Recall the set $S=\{H_1, H_2, \cdots, H_p\}$ and $H_i, H_j \in S$ if and only if $i\not\equiv j (mod p)$. Then there are three possibilities to consider: $i$ vectors in $S$ and $j$ vectors in $\cup_{j=1}^p\{H_p^j\}$; $i-1$ vectors in $S$, one vector in $\{H^p, H^{p^2}\}$  and $j$ vectors in $\cup_{j=1}^p\{H_p^j\}$; or $i$ vectors in $S$, one vector in $\{H^p, H^{p^2}\}$ and $j-1$ vectors in $\cup_{j=1}^p\{H_p^j\}$, where $i+j=q$. Altogether, by principals of addition and multiplication, there are\\ 
${{p}\choose{i}} {{p}\choose{p-1}}^{p-i} {{2}\choose{1}} {{p}\choose{j}} + {{p}\choose{i-1}} {{p}\choose{p-1}}^{p-(i-1)} {{2}\choose{2}} {{p}\choose{j}} + {{p}\choose{i}} {{p}\choose{p-1}}^{p-i} {{2}\choose{2}} {{p}\choose{j-1}}$
possibilities of resolving sets of cardinality  $\beta(\Gamma)+q$, where $i+j=q$.\\
Second, to find $r_{\beta(\Gamma)+k}$, where $p+1 \leq k\leq 2p-1$. Take the set of vertices $v_1, v_2, \cdots, v_k$ in $\Gamma$ that do not belong to any resolving set of cardinality $\beta(\Gamma)+k-1$. Since $k>p$, then we assume that $k=k_1+k_2$ such that $k_1 \neq k_2, k_1-1 \neq k_2$ and $k_1 \neq k_2-1$, where $1 \le k_j \le p$ for $j=1, 2$. Then there are the following possibilities:

$k_1$ vertices of the set $\{v_1, v_2, \cdots, v_k\}$ are in $S$ and $k_2$ vertices of the set $\{v_1, v_2, \cdots, v_k\}$ are in $\cup_{j=1}^p\{H_p^j\}$,

$k_2$ vertices of the set $\{v_1, v_2, \cdots, v_k\}$ are in $S$ and $k_1$ vertices of the set $\{v_1, v_2, \cdots, v_k\}$ are in $\cup_{j=1}^p\{H_p^j\}$,

$k_1$ vertices of the set $\{v_1, v_2, \cdots, v_k\}$ are in $S \cup \{H^p, H^{p^2}\}$ and $k_2$ vertices of the set $\{v_1, v_2, \cdots, v_k\}$ are in $\cup_{j=1}^p\{H_p^j\}$,

$k_1$ vertices of the set $\{v_1, v_2, \cdots, v_k\}$ are in $S$ and $k_2$ vertices of the set $\{v_1, v_2, \cdots, v_k\}$ are in $\cup_{j=1}^p\{H_p^j\} \cup \{H^p, H^{p^2}\}$,

$k_2$ vertices of the set $\{v_1, v_2, \cdots, v_k\}$ are in $S \cup\{H^p, H^{p^2}\}$ and $k_1$ vertices of the set $\{v_1, v_2, \cdots, v_k\}$ are in $\cup_{j=1}^p\{H_p^j\}$ or

$k_2$ vertices of the set $\{v_1, v_2, \cdots, v_k\}$ are in $S$ and $k_1$ vertices of the set $\{v_1, v_2, \cdots, v_k\}$ are in $\cup_{j=1}^p\{H_p^j\}\cup \{H^p, H^{p^2}\}$.\\
Again, by the principal of addition and multiplication, there are

${{p}\choose{k_1}} {{p}\choose{p-1}}^{p-k_1} {{2}\choose{1}} {{p}\choose{k_2}} + {{p}\choose{k_2}} {{p}\choose{p-1}}^{p-k_2} {{2}\choose{1}} {{p}\choose{k_1}} + {{p}\choose{k_1-1}} {{p}\choose{p-1}}^{p-(k_1-1)} {{2}\choose{2}} {{p}\choose{k_2}} + {{p}\choose{k_1}} {{p}\choose{p-1}}^{p-k_1} {{2}\choose{2}} {{p}\choose{k_2-1}} + {{p}\choose{k_2-1}} {{p}\choose{p-1}}^{p-(k_2-1)} {{2}\choose{2}} {{p}\choose{k_1}} + {{p}\choose{k_2}} {{p}\choose{p-1}}^{p-k_2} {{2}\choose{2}} {{p}\choose{k_1-1}}$\\ 
possible resolving sets of cardinality $\beta(\Gamma)+k$, where $p < k \leq 2p-1$.

By Lemma \ref{lem43}, $r_{\beta(\Gamma)+2p}=p^2+p+1$ and $r_{\beta(\Gamma)+2p+1}=1$.   
\end{proof}

\medskip
In the following remark, some additional possibilities of $r_{\beta(\Gamma)+k}$, where $p < k \leq 2p-1$, are given.

\medskip
 
\begin{remark}
In Theorem \ref{thm45}, we have the following additional possibilities:
\begin{enumerate}
\item if $k_1 = k_2$, then $r_{\beta(\Gamma)+k} =\\ {{p}\choose{k_1}}{{p}\choose{p-1}}^{p-k_1}{{2}\choose{1}}{{p}\choose{k_2}} + {{p}\choose{k_1-1}}{{p}\choose{p-1}}^{p-(k_1-1)}{{2}\choose{2}}{{p}\choose{k_2}} + {{p}\choose{k_1}}{{p}\choose{p-1}}^{p-k_1}{{2}\choose{2}}{{p}\choose{k_2-1}}$,

\item if $k_1-1 = k_2$, then $r_{\beta(\Gamma)+k} =\\ {{p}\choose{k_1}}{{p}\choose{p-1}}^{p-k_1}{{2}\choose{1}}{{p}\choose{k_2}} + {{p}\choose{k_2}}{{p}\choose{p-1}}^{p-k_2}{{2}\choose{1}}{{p}\choose{k_1}} + {{p}\choose{k_1-1}}{{p}\choose{p-1}}^{p-(k_1-1)}{{2}\choose{2}}{{p}\choose{k_2}} + {{p}\choose{k_1}}{{p}\choose{p-1}}^{p-k_1}{{2}\choose{2}}{{p}\choose{k_2-1}}
+ {{p}\choose{k_2-1}}{{p}\choose{p-1}}^{p-(k_2-1)}{{2}\choose{2}}{{p}\choose{k_1}}$, and

\item if $k_1 = k_2-1$, then $r_{\beta(\Gamma)+k} =\\
{{p}\choose{k_1}}{{p}\choose{p-1}}^{p-k_1}{{2}\choose{1}}{{p}\choose{k_2}} + {{p}\choose{k_2}}{{p}\choose{p-1}}^{p-k_2}{{2}\choose{1}}{{p}\choose{k_1}} + {{p}\choose{k_1-1}}{{p}\choose{p-1}}^{p-(k_1-1)}{{2}\choose{2}}{{p}\choose{k_2}} + {{p}\choose{k_1}}{{p}\choose{p-1}}^{p-k_1}{{2}\choose{2}}{{p}\choose{k_2-1}}
+ {{p}\choose{k_2}}{{p}\choose{p-1}}^{p-k_2}{{2}\choose{2}}{{p}\choose{k_1-1}}$.
\end{enumerate}
\end{remark}

\end{document}